\documentclass[11pt,a4paper]{article}

\RequirePackage{amsmath}
\RequirePackage{amssymb,latexsym}
\RequirePackage{amsthm}
\RequirePackage[hang,flushmargin,bottom,stable]{footmisc}
\RequirePackage[margin=0.35in,format=hang,font=small,labelfont=bf]{caption}
\RequirePackage{needspace}
\RequirePackage{graphicx,tikz}
\RequirePackage[T1]{fontenc}
\RequirePackage[sc]{mathpazo}
\RequirePackage[colorlinks=true,urlcolor=blue,linkcolor=blue,citecolor=blue]{hyperref}

\hypersetup{
pdfstartview={XYZ null null 1.00},
pdftitle={Thresholds for patterns in random permutations},
pdfauthor={David Bevan and Dan Threlfall},
pdfdisplaydoctitle=true,
pdflang={en-GB}
}

\newtheorem{thm}{Theorem} %[section]
\newtheorem{obs}[thm]{Observation}
\newtheorem{prop}[thm]{Proposition}
\newtheorem{cor}[thm]{Corollary}
\newtheorem{conj}[thm]{Conjecture}
\newtheorem{question}[thm]{Question}

\newenvironment{bullets} {\vspace{-9pt}\begin{itemize}\itemsep0pt} {\end{itemize}\vspace{-9pt}}
\newenvironment{myquote} {\vspace{-9pt}\begin{quote}} {\end{quote}\vspace{-9pt}}

\newcommand{\av}{\mathsf{Av}}
\newcommand{\bbE}{\mathbb{E}}
\newcommand{\bbP}{\mathbb{P}}
\newcommand{\C}{\mathbf{C}}
\newcommand{\CCC}{\mathcal{C}}
\newcommand{\ccnm}{\CCC_{n,m}}
\newcommand{\ceil}[1]{\left\lceil #1 \right\rceil}
\newcommand{\cnm}[1][m]{\C_{n,#1}}
\newcommand{\cnp}[1][p]{\C_{n,#1}}
\newcommand{\cov}[1]{\mathsf{Cov}\big[#1\big]}
\newcommand{\e}{\mathbf{e}}
\newcommand{\EEE}{\mathcal{E}}
\newcommand{\eenm}{\EEE_{n,m}}
\newcommand{\enm}[1][m]{\e_{n,#1}}
\newcommand{\enp}[1][p]{\e_{n,#1}}
\newcommand{\eq}{\mathchoice{\;=\;}{=}{=}{=}}
\newcommand{\expec}[1]{\bbE\big[#1\big]}
\newcommand{\floor}[1]{\left\lfloor #1 \right\rfloor}
\newcommand{\G}{\mathbf{G}}
\newcommand{\geqs}{\geqslant}
\newcommand{\gnm}[1][m]{\G_{n,#1}}
\newcommand{\gnp}[1][p]{\G_{n,#1}}
\DeclareMathOperator{\inv}{\mathsf{inv}}
\newcommand{\leqs}{\leqslant}
\newcommand{\liminfty}[1][n]{\lim\limits_{#1\to\infty}}
\newcommand{\prob}[1]{\bbP\big[#1\big]}
\newcommand{\QQQ}{\mathcal{Q}}
\newcommand{\s}{\boldsymbol{\sigma}}
\newcommand{\snm}[1][m]{\s_{n,#1}}
\newcommand{\snp}[1][p]{\s_{n,#1}}
\newcommand{\ssnm}{\SSS_{n,m}}
\newcommand{\SSS}{\mathcal{S}}
\newcommand{\together}[1]{\Needspace*{#1\baselineskip}}
\newcommand{\var}[1]{\mathrm{Var}\big[#1\big]}
\newcommand{\veps}{\varepsilon}

\newcommand{\plotc}[2]  % {length}{composition}
{
  \draw[thick] (0,0)--(#1,0);
  \foreach \y [count=\x] in {#2}
  {
    \draw (\x-.5,-1.15) node{\footnotesize\y};
    \ifnum0=\y {} \else {
      \filldraw[thick,fill=gray!25] (\x-1,0) rectangle (\x,\y);
    } \fi
  }
}

\newcommand{\plotptradius}{0.275}

\newcommand{\plotpt}[3][] % [colour]{x}{y}
{ \fill[#1,radius=\plotptradius] (#2,#3) circle; }

\newcommand{\plotpermnobox}[3][]  % [colour]{length-UNUSED}{perm}
{
  \foreach \y [count=\x] in {#3}
  {
    \ifnum0=\y {} \else {
      \plotpt[#1]{\x}{\y}
    } \fi
  }
}

\newcommand{\plotperm}[3][]  % [colour]{length}{perm}
{
  \plotpermnobox[#1]{#2}{#3}
  \draw[thick] (.5,.5) rectangle (#2.5,#2.5);
}

\newcommand{\plotpermnums}[4][]  % [colour]{length}{perm}{nums}
{
  \plotperm[#1]{#2}{#3}
  \foreach \y [count=\x] in {#4}
  {
    \draw (\x,-.5) node{\footnotesize\y};
  }
}

\newcommand{\plotpermgrid}[3][]  % [colour]{length}{perm}
{
  \foreach \x in {1,...,#2} \draw[very thin] (\x,.5)--(\x,#2.5) (.5,\x)--(#2.5,\x);
  \plotperm[#1]{#2}{#3}
}

\newcommand{\circpt}[3][] % [colour]{x}{y}
{ \draw[#1,radius=\plotptradius+0.075] (#2,#3) circle; }

\title{\textbf{Thresholds for patterns in random permutations\\with a given number of inversions}}
\author{David Bevan${}^\dagger$ and Dan Threlfall${}^{\dagger\ddag}$}
\date{}

\begin{document}
\maketitle

{\begin{NoHyper}
\let\thefootnote\relax\footnotetext
{${}^\dagger$Department of Mathematics and Statistics, University of Strathclyde, Glasgow, Scotland.}
\let\thefootnote\relax\footnotetext
{${}^\ddag$The second author was supported by an EPSRC Mathematical Sciences Studentship.}
\end{NoHyper}}

{\begin{NoHyper}
\let\thefootnote\relax\footnotetext
{
%05C80 % Random graphs (graph-theoretic aspects)
%05A15 % Exact enumeration problems, generating functions
%05A16 % Asymptotic enumeration
}
\end{NoHyper}}

\begin{abstract}
\noindent
We explore how the asymptotic structure of a random permutation of $[n]$ with $m$ inversions evolves, as $m$ increases,
establishing thresholds for the appearance and disappearance of any classical, consecutive or vincular pattern.
The threshold for the appearance of a classical pattern depends on the greatest number of inversions in any of its sum indecomposable components.
%We also calculate the probability of a pattern occurring at its threshold.
\end{abstract}

% ================================================================
\section{Introduction}\label{sectIntro}

We consider permutations
from an evolutionary perspective, in an analogous manner to the Gilbert--Erd\H{o}s--R\'enyi random graph~\cite{ER1959,ER1960,Gilbert1959}.
Our model, which we call the \emph{uniform random permutation}, and denote $\snm$, is a permutation drawn uniformly from the set of permutations of $[n]$ with exactly $m$ inversions.
We are interested in how, for large $n$, the structure of $\snm$ evolves as the number of its inversions $m$ increases from zero to~$\binom{n}2$.
Specifically, for any classical, consecutive or vincular pattern $\pi$, we establish thresholds for the
appearance and disappearance of $\pi$ in~$\snm$.
These results build on our previous work~\cite{BTCompositions} on thresholds for patterns in random compositions.

A \emph{permutation} or \emph{n-permutation} is considered to be simply an arrangement of the numbers $[n]:=\{1, 2,\ldots , n\}$ for some positive $n$.
Let $\SSS_n$ denote the set of all $n$-permutations.
We often display an $n$-permutation $\sigma$ using its \emph{plot}, the set of points $(i,\sigma(i))$ in the Euclidean plane, for $i = 1,\dots ,n$.
Sometimes we identify a permutation with its plot.
If $\sigma$ is an $n$-permutation, we define its \emph{complement},  denoted $\overline{\sigma}$, to be the permutation such that $\overline{\sigma}(i) = n + 1 - \sigma(i)$ for every $i \in [n]$.
Thus the plot of $\overline{\sigma}$ is the reflection of the plot of $\sigma$ about a horizontal axis.
See Figure~\ref{figPermPlot} for the plots of an $8$-permutation and its complement.

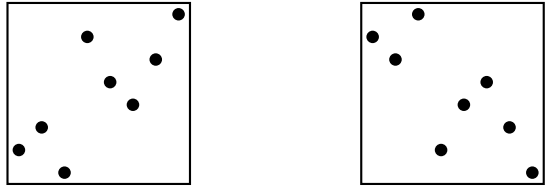
\begin{figure}[t]
  \centering
  \begin{tikzpicture}[scale=0.3]
 %\plotperm{9}{1,3,2,7,5,4,6,9,8}
  \plotperm{8}{2,3,1,7,5,4,6,8}
  \end{tikzpicture}
   \hspace{2cm}
  \begin{tikzpicture}[scale=0.3]
 %\plotperm{9}{9,7,8,3,5,6,4,1,2}
  \plotperm{8}{7,6,8,2,4,5,3,1}
  \end{tikzpicture}
  \caption{The permutation $\sigma = 23175468$ and its complement $\overline{\sigma} = 76824531$}\label{figPermPlot}
\end{figure}

%\together3
An \emph{inversion} in a permutation $\sigma  \in \SSS_n$ is a pair $i,j \in [n]$ such that $i < j$ and $\sigma(i) > \sigma(j)$.
We use $\inv(\sigma)$ to denote the total number of inversions in the permutation $\sigma$,
and use \[\ssnm \eq \{\sigma \in \SSS_n \::\: \inv(\sigma) = m\}\] to denote the set of $n$-permutations with exactly $m$ inversions.
Thus $\boldsymbol\sigma_{n,m}$ is a permutation chosen uniformly from $\ssnm$.
The greatest possible number of inversions that can occur in an $n$-permutation is $\binom{n}{2}$.
Note that $\inv(\overline{\sigma}) = \binom{n}{2} - \inv(\sigma)$.
Thus, $\snm[\binom{n}2-m]$ has the same distribution as $\overline{\snm}$.

We consider three different forms of permutation pattern containment.
For a very brief introduction to permutation patterns, see~\cite{BevanPPBasics}; for more extended expositions, see either B\'ona~\cite{BonaBook} or Kitaev~\cite{KitaevBook}.

A $k$-permutation $\pi$ occurs as a \emph{consecutive pattern} at position $j$ in a permutation $\sigma$ if the consecutive subsequence $\sigma(j)\ldots\sigma(j+k-1)$ has the same relative ordering as $\pi$.
%A permutation $\pi$ occurs as a \emph{consecutive pattern} in a permutation $\sigma$ if $\sigma$ has a consecutive subsequence whose terms have
For example, the consecutive pattern $132$ occurs three times in the permutation at the left of Figure~\ref{figPermPlot}, at positions 1, 3 and~7.
A permutation that doesn't contain a pattern is said to \emph{avoid} it.
See~\cite{Elizalde2016,EN2003,EN2012} for investigations of permutations avoiding consecutive patterns.

A permutation $\pi$ occurs as a \emph{classical pattern} in $\sigma$ if $\sigma$ has a (not necessarily consecutive) subsequence whose terms have the same relative ordering as $\pi$.
For example, the classical pattern $312$ occurs twice in the permutation at the left of Figure~\ref{figPermPlot},
one occurrence consisting of the points
at positions 4, 5 and~7, and the other consisting of the points
at positions 4, 6 and~7.
For an extensive survey on classical patterns, see Vatter~\cite{VatterSurvey}.

Finally, in a \emph{vincular pattern} (see~\cite{BS2000,BP2012,BCKPosets,BCDK2010,Claesson2001,Elizalde2006,Hofer2018}) only some terms  are required to be adjacent.
Consecutive terms in a vincular pattern that must be adjacent are underlined.
For example, the vincular patterns $\underline{31}2$ and $3\underline{12}$ each occur once in the permutation at the left of Figure~\ref{figPermPlot},
the former consisting of the points
at positions 4, 5 and~7, and the latter consisting of the points
at positions 4, 6 and~7.

We take a dynamic (or evolutionary) view by considering a process on $n$-permutations, namely a sequence of permutations $\sigma_{0},\sigma_{1},\sigma_{2},\dots,\sigma_{\binom{n}{2}}$, where $\sigma_{t+1}$ is obtained from $\sigma_{t}$, by the addition of one inversion.
In this context, a striking phenomenon is the abrupt appearance and disappearance of substructures.
To quantify this, we introduce the concept of a threshold.
A function $m^\star = m^\star(n)$ is a threshold in $\snm$ for
(the appearance of)
a property $\QQQ$ of permutations if
\[
\liminfty \prob{\snm \text{ satisfies }\QQQ} \eq
\begin{cases}
0 & \text{if~ $m\ll m^\star$,} \\
1 & \text{if~ $m^\star\ll m\ll m^+$,}
\end{cases}
\]
for some function $m^+\gg m^\star$,
where, throughout this work,
we write $f \ll g$ to denote that $\liminfty f/g =0$.
We also write $f \sim g$ if $\liminfty f/g =1$.

We say that a property holds \emph{asymptotically almost surely} (a.a.s) if asymptotically the probability that it holds equals $1$.
Thus, above its threshold, %$\QQQ$ holds asymptotically almost surely,
a.a.s. $\QQQ$ holds,
whereas below its threshold, a.a.s. $\QQQ$ does not hold.
%At the threshold, that is when $m\sim c m^\star$ for some constant~$c$, the asymptotic probability that $\snm$ satisfies $\QQQ$ may lie strictly between 0 and~1.

With a slight abuse of terminology, we also say that $\binom{n}2-m\sim m^\star$ is a threshold in $\snm$ for
the \emph{disappearance} of
a property $\QQQ$ if
\[
\liminfty \prob{\snm \text{ satisfies }\QQQ} \eq
\begin{cases}
    1  & \text{if~ $m^\star \ll \binom{n}{2} - m \ll m^+$,}\\
    0  & \text{if~ $\binom{n}{2} - m \ll m^\star$,}
  \end{cases}
\]
for some function $m^+\gg m^\star$.
We determine the thresholds for the appearance and disappearance of patterns in~$\snm$.  
%and also the asymptotic probability of containing a pattern at a threshold.

% --------------------------------
\subsection*{Consecutive patterns}
%\together5
%\textbf{Consecutive patterns}

For a \emph{consecutive} pattern, these thresholds depend on the number of inversions in the pattern, and on the number of inversions in its complement, respectively.
Specifically, we have the following result.

%\together9
\begin{thm}\label{thmMainResultConsec}
Let $\pi$ be any consecutive permutation pattern of length $k$.
If $s=\inv(\pi)$ and $s'=\inv(\overline{\pi})$, then for any positive constant~$a$,
\begin{align*}
\liminfty\prob{\snm \text{ contains }\pi} &\eq \begin{cases}
    0 & \quad \text{if~ } m \ll n^{1-1/s},\\
    1-e^{-a^{s}} & \quad \text{if~ $m\sim a n^{1-1/s}$} \text{~and~} s>1,\\
    1 & \quad \text{if~ } m \sim a \text{~and~} s=1,\\
    1 & \quad \text{if~ } n^{1 - 1/s} \ll m \ll n^{1+1/k},
  \end{cases}
\\[3pt]
\liminfty\prob{\snm \text{ contains }\pi} &\eq \begin{cases}
    1 & \quad  \text{if~ } n^{1+1/k} \gg \binom{n}{2} - m \gg n^{1-1/s'},\\
    1 & \quad \text{if~ } \binom{n}{2} - m \sim a \text{~and~} s=1,\\
    1-e^{-a^{s'}} & \quad \text{if~ $\binom{n}{2} - m\sim a n^{1-1/s'}$} \text{~and~} s>1, \\
    0 & \quad \text{if~ } \binom{n}{2} - m \ll n^{1-1/s'},
  \end{cases}
\end{align*}
as long as $s>0$ and $s'>0$, respectively.
\end{thm}

For example, the threshold for the appearance of consecutive pattern $2143$ in $\snm$ is $m\sim \sqrt{n}$, and the threshold for its disappearance is $\binom{n}2-m\sim n^{3/4}$.

Unfortunately, our methods do not enable us to show that $\snm$ a.a.s. contains a given pattern for all values of $m$ between the thresholds for its appearance and disappearance.
We defer further discussion of this challenge to Section~\ref{sectGapMallows}.

\begin{figure}[ht]
\centering
\begin{tikzpicture}[scale=0.3]
\plotperm{8}{2,3,1,7,5,4,6,8}
{\draw[thin] (0.5,0.5) rectangle (3.5,3.5);}
{\draw[thin] (3.5,3.5) rectangle (7.5,7.5);}
{\draw[thin] (7.5,7.5) rectangle (8.5,8.5);}
\end{tikzpicture}
\hspace{2cm}
\begin{tikzpicture}[scale=0.3]
\plotperm{8}{7,6,8,2,4,5,3,1}
\end{tikzpicture}
\caption{The sum decomposition of a decomposable permutation, and an indecomposable permutation}
\label{figDecompPerm}
\end{figure}
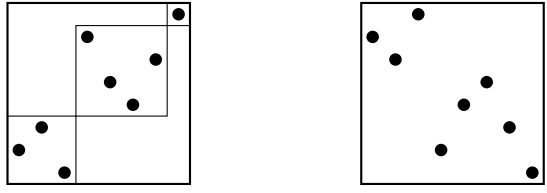

% --------------------------------
\subsection*{Classical patterns}
%\together5
%\textbf{Classical patterns}

Given an $n$-permutation $\sigma$, we say that it is \emph{decomposable} if there exists some $k < n$ such that
\[\{\sigma(1),\sigma(2),\dots,\sigma(k)\} \eq \{1,2,\dots,k\}.\]
If a permutation is not decomposable, we say it is \emph{indecomposable}.
For example, Figure~\ref{figDecompPerm} displays the plot of a decomposable permutation at the left and the plot of an indecomposable permutation at the right.
Any permutation that is decomposable can be expressed as the combination of two or more shorter permutations. Given two permutations $\sigma$ and $\tau$ with lengths $k$ and $\ell$ respectively, their \emph{direct sum} $\sigma \oplus \tau$ is the permutation of length $k + \ell$ consisting of $\sigma$ followed by a shifted copy of~$\tau$:
\[(\sigma \oplus \tau)(i)\eq\begin{cases}
\sigma(i) \qquad &\text{ if~ } i \leqs k,\\
k + \tau(i-k) &\text{ if~ } k+1 \leqs i \leqs k + \ell.
\end{cases}\]
For example, the permutation at the left of Figure~\ref{figDecompPerm} is $231 \oplus 4213 \oplus 1 $.
Every permutation has a unique representation as the direct sum of a sequence of one or more indecomposable permutations, which we call its \emph{components}.
This representation is known as its \emph{sum decomposition}.
Note that the complement of a decomposable permutation (with more than one component) is indecomposable (having only one component),
as illustrated in Figure~\ref{figDecompPerm}:
the permutation at the right is the complement of the permutation at the left.

The threshold for the appearance of a \emph{classical} pattern depends on the greatest number of inversions in one of its components, and
the threshold for its disappearance depends on the greatest number of inversions in a component of its complement.

%\together9
\begin{thm}\label{thmMainResultClassical}
Let $\pi$ be any classical permutation pattern.
If $s$ is the greatest number of inversions in a component of $\pi$,
and $s'$ is the greatest number of inversions in a component of $\overline{\pi}$,
then
\begin{align*}
\liminfty\prob{\snm \text{ contains }\pi} &\eq \begin{cases}
    0 & \quad \text{if~ } m \ll n^{1-1/s},\\
    1 & \quad \text{if~ } n^{1 - 1/s} \ll m \ll n,
  \end{cases}
\\[3pt]
\liminfty\prob{\snm \text{ contains }\pi} &\eq \begin{cases}
    1  & \quad  \text{if~ } n \gg \binom{n}{2} - m \gg n^{1-1/s'},\\
    0       & \quad \text{if~ } \binom{n}{2} - m \ll n^{1-1/s'},
  \end{cases}
\end{align*}
as long as $s>0$ and $s'>0$, respectively.
\end{thm}

For example, the threshold for the appearance of classical pattern $23175468=231\oplus 4213 \oplus 1$ (shown at the left of Figure~\ref{figDecompPerm}) in $\snm$ is $m=n^{3/4}$, since component 4213 has four inversions and neither of the other two components have more.

% --------------------------------
\subsection*{Vincular patterns}
%\together5
%\textbf{Vincular patterns}
\label{defVinc}

A \emph{vincular} pattern with sum decomposition $\alpha_1\oplus\ldots\oplus\alpha_k$,
has a unique (possibly coarser) representation as a direct sum $\beta_1\oplus\ldots\oplus\beta_\ell$ for some $\ell\leqs k$, such that
\begin{bullets}
\item[(a)] each $\beta_j=\alpha_{i_j}\oplus\alpha_{i_j+1}\oplus\ldots\oplus\alpha_{i_j+r_j}$ for some $i_j$ and $r_j$, and
\item[(b)] $\alpha_i$ and $\alpha_{i+1}$ are components of the same $\beta_j$ only if the last term of $\alpha_i$ is required to be adjacent to the first term of $\alpha_{i+1}$.
\end{bullets}
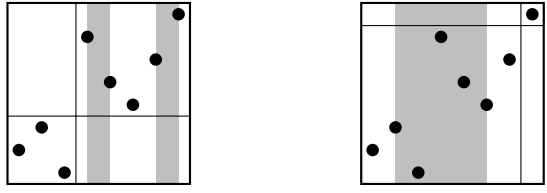
\begin{figure}[t]
\centering
\begin{tikzpicture}[scale=0.3]
\fill[gray!50] (4,.5) rectangle (5,8.5);
\fill[gray!50] (7,.5) rectangle (8,8.5);
\plotperm{8}{2,3,1,7,5,4,6,8}
{\draw[thin] (0.5,0.5) rectangle (3.5,3.5);}
{\draw[thin] (3.5,3.5) rectangle (8.5,8.5);}
\end{tikzpicture}
\hspace{2cm}
\begin{tikzpicture}[scale=0.3]
\fill[gray!50] (2,.5) rectangle (6,8.5);
\plotperm{8}{2,3,1,7,5,4,6,8}
{\draw[thin] (0.5,0.5) rectangle (7.5,7.5);}
{\draw[thin] (7.5,7.5) rectangle (8.5,8.5);}
\end{tikzpicture}
\caption{The supercomponents of vincular patterns $231\underline{75}4\underline{68}$ and $2\underline{31754}68$}
\label{figSupercomp}
\end{figure}
We say that $\beta_1,\ldots,\beta_\ell$ are the pattern's \emph{supercomponents}.
For example,
considering the permutation %$23175468$
shown at the left of Figure~\ref{figDecompPerm},
the vincular pattern
$231\underline{75}4\underline{68}$ has supercomponent decomposition $231 \oplus 42135$, whereas $2\underline{31754}68$ decomposes as $2317546 \oplus 1$.
See Figure~\ref{figSupercomp} for an illustration, in which the adjacency criteria are shown by shading.

The threshold for the appearance of a vincular pattern depends on the greatest number of inversions in one of its supercomponents.

\begin{thm}\label{thmMainResultVinc}
Let $\pi$ be any vincular permutation pattern.
If $s$ is the greatest number of inversions in a supercomponent of $\pi$,
and $s'$ is the greatest number of inversions in a supercomponent of $\overline{\pi}$,
then
\begin{align*}
\liminfty\prob{\snm \text{ contains }\pi} &\eq \begin{cases}
    0 & \quad \text{if } m \ll n^{1-1/s},\\
    1 & \quad \text{if } n^{1 - 1/s} \ll m \ll n,
  \end{cases}
\\[3pt]
\liminfty\prob{\snm \text{ contains }\pi} &\eq \begin{cases}
    1  & \quad  \text{if } n \gg \binom{n}{2} - m \gg n^{1-1/s'},\\
    0  & \quad \text{if } \binom{n}{2} - m \ll n^{1-1/s'},
  \end{cases}
\end{align*}
as long as $s>0$ and $s'>0$, respectively.
\end{thm}

% --------------------------------
\subsection*{Background}
%\together5

There has not been a great deal of previous study of the structure of permutations with a given number of inversions.
Even the magnitude of $|\ssnm|$ appears not to have been established for all ranges of~$m$.
{Comtet}~\cite[Section~7.4]{Comtet1974} proves that the number of inversions in a random $n$-permutation satisfies a central limit theorem
%(see also {Bender}~\cite[Example~5.5]{Bender1973} and {Fulman}~\cite{Fulman2004}).
(see also~\cite{Bender1973,Fulman2004}).
Asymptotics for $|\ssnm|$ have been determined when $m\leqs n$ (see~\cite{Margolius2001}), when $m\sim a n$ (see~\cite{Clark2000,LP2003}), and when $m\sim a n^2$ (see~\cite{LP2003}).
%{Clark}~\cite{Clark2000} gives a complete asymptotic expansion when $m\sim n^2/4$.
%{Margolius}~\cite{Margolius2001} presents asymptotics for $|\ssnm|$ when $m\leqs n$.
%{Louchard and Prodinger}~\cite{LP2003} extend this to $m\sim a n$.
%They also consider the dense regime $m\sim a n^2$.
However, the gap $n\ll m\ll n^2$ seems not have been investigated.

Apart from the flawed preprint~\cite{BevanLocallyUniform}, the only prior work specifically on $\snm$
of which we are aware
is that of Acan and Pittel~\cite{AP2013}.
Their primary result is a determination of the (sharp) threshold at which $\snm$ becomes indecomposable --- at $m\sim(6/\pi^2)n\log n$. % at $m=(6/\pi^2)n\log n$.
They make use of an implicitly defined Markov process that produces $\snm[m+1]$ from $\snm$. %, in which the transition probabilities depend on the {whole} %inversion sequence of the
%permutation. % (unlike for random graphs).
No {explicit} model of this evolutionary process is known.
{Kenyon, Kr\'al', Radin and Winkler}~\cite{KKRW2020} %study permutons constrained by having fixed densities of a finite number of (classical) patterns and
compute the limit shapes of permutations when $m\sim a n^2$, % $m\sim\rho\binom{n}2$, % (see their Figure~1 and Section~6),
thus making it possible to determine the expected density of any classical pattern in $\snm$ in this range.

The structure of a random permutation $\s_n$ drawn uniformly from $\SSS_n$ has been rather better studied.
We mention just a few results.
Janson, Nakamura and Zeilberger~\cite{JNZ2015} establish asymptotic normality for the distribution of any classical pattern, a result which has been extended to every vincular pattern by Hofer~\cite{Hofer2018}.
Perarnau~\cite{Perarnau2013} investigates consecutive pattern avoidance in $\s_n$.
{Bhattacharya and Mukherjee}~\cite{BM2017} determine the number of inversions involving a given point of $\s_n$, proving convergence to a uniform distribution (over a range dependent on the position) except in the case of the central point (which satisfies a central limit theorem).
Probably the most celebrated result in this context is the establishment by {Baik, Deift and Johansson}~\cite{BDJ1999} of the limiting distribution of the length of the longest increasing subsequence in a random $n$-permutation. See~\cite{Romik2015} for an extended exposition.

% --------------------------------
\subsection*{Outline}
%\textbf{Outline}
In Section~\ref{sectPermsInvSeqs}, we consider inversion sequences of permutations, relating the presence of a consecutive pattern in a permutation to the inversion sequences of the pattern and of the host permutation.
In Section~\ref{sectCompositions}, we apply our work on patterns in random compositions~\cite{BTCompositions} to inversion sequences and prove Theorem~\ref{thmMainResultConsec} on the thresholds for consecutive patterns.
Section~\ref{sectClassicalVinc} builds on this to prove Theorems~\ref{thmMainResultClassical} and~\ref{thmMainResultVinc} giving the thresholds for classical and vincular patterns.
Various consequences of these theorems are discussed.
Finally, in Section~\ref{sectGapMallows}, we present several open questions, including considering the challenge of bridging the gap between the thresholds for a pattern's appearance and its disappearance, and briefly discussing the relationship between $\snm$ and Mallows permutations.

% ================================================================
\section{Permutations and inversion sequences}\label{sectPermsInvSeqs}

In the section, we introduce the representation of permutations as inversion sequences and investigate the relationship between the
containment of consecutive patterns in the
two representations.
We start with the observation that the distribution of any consecutive pattern in $\snm$ is independent of its position.
This holds for any given $n$ and $m$.
As a consequence, in subsequent arguments, we need only consider the occurrence of patterns at position 1.
This proposition first appeared in the unpublished preprint~\cite{BevanLocallyUniform}.

\begin{prop}\label{propSnmShiftInvariance}%[Shift invariance of $\ssnm$]\ref{propSnmShiftInvariance}
  For any consecutive permutation pattern $\pi$ of length $k$
  and any %positive
  $i,j\in [n+1-k]$,
  \[
  \prob{\text{$\pi$ occurs at position $i$ in $\snm$}} \:=\: \prob{\text{$\pi$ occurs at position $j$ in $\snm$}}
  .
  \]
\end{prop}

This result follows from the existence of an operation that removes the last point from a permutation and adds a new first point in such a way as to preserve the number of inversions. This operation shifts %consecutive
patterns rightwards.

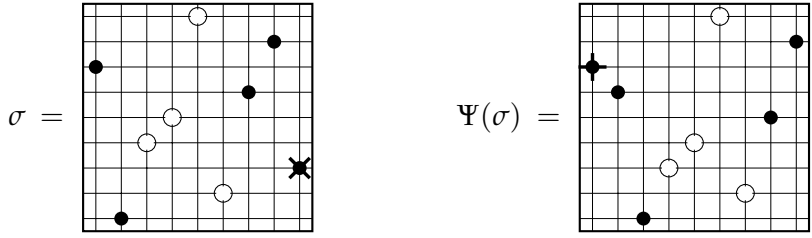
\begin{figure}[t] %[ht]
\DeclareRobustCommand{\boldplus}{\mathord{\begin{tikzpicture}\draw[line width=0.25ex, x=1.5ex, y=1.5ex] (0.5,-0.05) -- (0.5,1.05)(-0.05,0.5) -- (1.05,0.5);\end{tikzpicture}}}
\DeclareRobustCommand{\boldtimes}{\mathord{\begin{tikzpicture}\draw[line width=0.25ex, x=1.5ex, y=1.5ex] (0.1,0.1) -- (0.9,0.9)(0.1,0.9) -- (0.9,0.1);\end{tikzpicture}}}
\centering
\begin{tikzpicture}[scale=0.335]
    \draw[ line width=0.25ex] (8.6,2.6) -- (9.4,3.4);
    \draw[ line width=0.25ex] (9.4,2.6) -- (8.6,3.4);
   %\plotpermgrid  [blue!50!black]{9}{7,1,0,0,0,0,6,8,3}
   %\plotpermnobox[green!50!black]{9}{0,0,4,5,9,2}
    \plotpermgrid        {9}{7,1,0,0,0,0,6,8,3}
    \plotpermnobox[white]{9}{0,0,4,5,9,2}
   %\plotpermgrid        {9}{7,1,4,5,9,2,6,8,3}
    \circpt{3}{4}
    \circpt{4}{5}
    \circpt{5}{9}
    \circpt{6}{2}
    \node[left] at (.7,5) {$\sigma \:=\: {}$};
\end{tikzpicture}
$\qquad\qquad$
\begin{tikzpicture}[scale=0.335]
    \draw[ line width=0.25ex] (0.45,7) -- (1.55,7);
    \draw[ line width=0.25ex] (1,6.45) -- (1,7.55);
   %\plotpermgrid  [blue!50!black]{9}{7,6,1,0,0,0,0,5,8}
   %\plotpermnobox[green!50!black]{9}{0,0,0,3,4,9,2}
    \plotpermgrid        {9}{7,6,1,0,0,0,0,5,8}
    \plotpermnobox[white]{9}{0,0,0,3,4,9,2}
   %\plotpermgrid        {9}{7,6,1,3,4,9,2,5,8}
    \circpt{4}{3}
    \circpt{5}{4}
    \circpt{6}{9}
    \circpt{7}{2}
    \node[left] at (.7,5) {$\Psi(\sigma) \:=\: {}$};
\end{tikzpicture}
  \caption{The bijection used in the proof of Proposition~\ref{propSnmShiftInvariance}:
  the point marked $\boldtimes$ is replaced by that marked $\boldplus$;
  the consecutive pattern 2341 occurs at position~3 in $\sigma$ and at position~4 in~$\Psi(\sigma)$}.
  \label{figShiftInvariance}
\end{figure}

\together7
\begin{proof}
  As illustrated in Figure~\ref{figShiftInvariance},
  let $\Psi:\mathcal{S}_{n,m}\to\mathcal{S}_{n,m}$ be defined by
  \[
  \Psi(\sigma) \eq
  \Psi(\sigma_1\sigma_2\ldots\sigma_n) \eq \sigma' \eq \sigma'_0\sigma'_1\ldots\sigma'_{n-1} ,
  \]
  where
  $\sigma'_0=n+1-\sigma_n$, and for $1\leqs i<n$,
  $$
  \sigma'_i \eq
  \begin{cases}
  \sigma_i+1 , & \text{~if~ $\sigma'_0\leqs\sigma_i<\sigma_n$,} \\
  \sigma_i-1 , & \text{~if~ $\sigma_n<\sigma_i\leqs \sigma'_0$,} \\
  \sigma_i   , & \text{~otherwise.}
  \end{cases}
  $$
Note that $\sigma_n$ contributes $n-\sigma_n$ inversions to $\sigma$, and
$\sigma'_0$ contributes the same number of inversions to $\sigma'$.
For $0< i< n$, the point $\sigma'_i$ contributes the same number of inversions to $\sigma'$ as
$\sigma_i$ does to $\sigma$. So $\inv(\sigma')=\inv(\sigma)$.
Since $\Psi$ preserves length and has a well-defined inverse, it is a bijection on $\mathcal{S}_{n,m}$.

If $\pi$ occurs at position $j\leqs n-k$ in $\sigma$, then $\pi$ occurs at position $j+1$ in $\Psi(\sigma)$.
Hence, if $1\leqs i,j \leqs n+1-k$, then $\pi$ occurs at position $i$ in $\sigma$ if and only if $\pi$ occurs at position $j$ in~$\Psi^{j-i}(\sigma)$.
\end{proof}

\begin{figure}[ht]
  \centering
  \begin{tikzpicture}[scale=0.3]
  \plotpermnums{9}{3,1,4,8,6,2,7,5,9}{0,1,0,0,1,4,1,3,0}
  \end{tikzpicture}
  \caption{A permutation and its inversion sequence}\label{figInvSeqAtStart}
\end{figure}
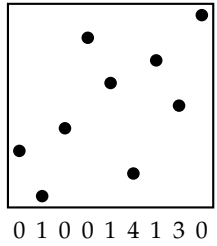

Key to our analysis is the representation of permutations as \emph{inversion sequences}.
Given an $n$-permutation $\sigma$, its {inversion sequence} $e_{\sigma}$ is the sequence of integers $\big(e_{\sigma}(j)\big)^{n}_{j=1}$, where
\[e_{\sigma}(j) \eq \big|\{i : i < j \text{ and } \sigma(i) > \sigma(j)\}\big|\]
is the number of inversions involving $\sigma(j)$ and the terms of $\sigma$ preceding $\sigma(j)$, or equivalently the number of points to the upper left of $(j,\sigma(j))$ in the plot of $\sigma$.
%For example, the inversion sequence of the permutation at the left of Figure~\ref{figShiftInvariance} is 011104216.
See Figure~\ref{figInvSeqAtStart} for an example.
Each permutation has a unique inversion sequence.
Clearly, for each $j$, we have $0\leqs e_\sigma(j)<j$, and in fact integer sequences satisfying this condition whose sum equals~$m$ are exactly the inversion sequences of $n$-permutations with $m$ inversions.
We use $\eenm$ to denote the set of such inversion sequences.

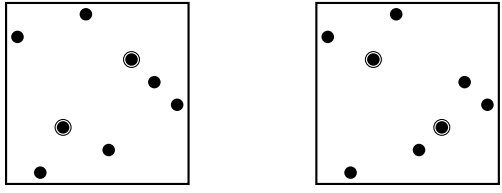
\begin{figure}[t]
\centering
\begin{tikzpicture}[scale=0.3]
\plotperm{8}{7,1,3,8,2,6,5,4}
\circpt{6}{6}
\circpt{3}{3}
\end{tikzpicture}
\qquad\qquad
\begin{tikzpicture}[scale=0.3]
\plotperm{8}{7,1,6,8,2,3,5,4}
\circpt{3}{6}
\circpt{6}{3}
\end{tikzpicture}
  \caption{The permutations $\sigma = 71382654$ and $\sigma' = 71682354$}
  \label{figAddInv}
\end{figure}

Given an inversion sequence $e$, if $e(j) < j-1$ then $e^{+j}$ denotes the inversion sequence obtained from $e$ by the addition of $1$ to its $j$th term.
Incrementing a term in the inversion sequence of a permutation just switches the values of two terms.
See Figure~\ref{figAddInv} for an example, in which $e_{\sigma} = 01103234$ and $e_{\sigma'} = e_{\sigma}^{+6} = 01103334$.

\begin{obs}\label{obsIncrPerm}
  Let $\sigma$ be a permutation.
  Suppose $e_{\sigma}(j) < j-1$, and that $\sigma'$ is the permutation with inversion sequence $e_\sigma^{+j}$.
  Let $i < j$ be the index such that
  \[
  \sigma(i) \eq \max\big\{\sigma(k) \::\: k < j \text{ ~and~ } \sigma(k) < \sigma(j)\big\} .
  \]
  Then, $\sigma'(i) = \sigma(j)$ and $\sigma'(j) = \sigma(i)$, and $\sigma'(k) = \sigma(k)$ for each $k\neq i,j$.
\end{obs}

If $0 \leqs m \leqs \binom{n}{2}$, then we use $\enm$ to denote an inversion sequence chosen uniformly from $\eenm$.
We call $\enm$ the \emph{uniform random inversion sequence.}
Since there is a bijection between $\eenm$ and $\mathcal{S}_{n,m}$,
we know that $\enm$ and $e_{\snm}$ have the same distribution.

If a consecutive permutation pattern $\pi$ occurs at a position $j \neq 1$ in a permutation $\sigma$, then it is not necessarily the case that $e_{\pi}$ occurs at position $j$ in $e_{\sigma}$.
However, if $\pi$ occurs at position $1$ in~$\sigma$, then $e_{\pi}$ does occur at position $1$ in $e_{\sigma}$, as we prove below.
For example, the consecutive pattern 213 occurs at positions 1, 5 and 7 in the permutation in Figure~\ref{figInvSeqAtStart}.
However $e_{213} = 010$ only occurs at position 1 (and not at positions 5 and 7) in its inversion sequence.

%\together6
\begin{prop}\label{propPos1inPermPos1inInv}
Let $\pi$ be any consecutive permutation pattern.
If $\pi$ occurs at position $1$ in a permutation $\sigma$, then $e_{\pi}$ occurs at position $1$ in $e_{\sigma}$.
\end{prop}
\begin{proof}
If $\pi$ has length $k$, then for each $j\in[k]$,
  \[
  e_{\sigma}(j) \eq \big|\{i : i < j \text{ and } \sigma(i) > \sigma(j)\}\big|
                \eq \big|\{i : i < j \text{ and } \pi(i) > \pi(j)\}\big|
                \eq e_\pi(j) .
                \qedhere
  \]
\end{proof}

In general, if a consecutive pattern $\pi$ occurs in $\sigma$, then the corresponding terms of $e_{\sigma}$ satisfy a chain of inequalities that depend only on $\pi$.
We defer the determination of the specific correspondence between patterns and inequalities
to Proposition~\ref{propIneqs} in an appendix, since this result is not needed to derive our main results.
However, we do require the following implication in the opposite direction to that in Proposition~\ref{propPos1inPermPos1inInv}.

\begin{prop}\label{propPosjinInvPosjinPerm}
Let $\pi$ be any consecutive permutation pattern.
If $\sigma$ is a permutation and $e_{\pi}$ occurs at position $j$ in $e_{\sigma}$, then $\pi$ occurs at position $j$ in~$\sigma$.
Moreover, if $\pi$ has length $k$, then for all $i<j$ and $\ell=j,\ldots,j+k-1$, we have $\sigma(i)<\sigma(\ell)$.
\end{prop}

\begin{proof}
  We proceed by induction on the length of the pattern.
  If $\pi$ has length 1, then $\pi=1$ and $e_\pi=0$.
  Hence, $e_\sigma(j)=0$, so there is no point in the plot of $\sigma$ to the upper left of $\sigma(j)$. %$(j,\sigma(j))$.

  Suppose now that the proposition holds for patterns of length less than $k$, and that $\pi$ has length~$k$.
  Let $\pi'$ be the permutation of length $k-1$ that results from the removal of the last point of~$\pi$.
  If $e_{\pi}$ occurs at position $j$ in $e_{\sigma}$ then $e_{\pi'}$ also occurs at position $j$ in~$e_{\sigma}$.
  So, by the induction hypothesis, $\pi'$ occurs at position $j$ in~$\sigma$, with
  no point of $\sigma$ to the upper left of any of the $k-1$ points $\sigma(j),\ldots,\sigma(j+k-2)$ that form its occurrence.

  Now $e_\pi(k)<k$.
  So at most $k-1$ points of $\sigma$ are to the upper left of $\sigma(j+k-1)$, all of which must therefore be part of the occurrence of~$\pi'$, forming an occurrence of $\pi$ at position~$j$ in~$\sigma$.
\end{proof}

%\HIDE
{
Propositions~\ref{propSnmShiftInvariance} and~\ref{propPos1inPermPos1inInv} immediately imply the following result.

\begin{prop}\label{propSnmtoEnm}
For any consecutive permutation pattern $\pi$ of length $k$
  and any
  $j\in [n+1-k]$,
\begin{align*}
\prob{\pi \text{ occurs at position } j \text{ in } \snm} &\eq \prob{\pi \text{ occurs at position } 1 \text{ in } \snm}\\
 &\eq \prob{e_{\pi} \text{ occurs at position } 1 \text{ in } \enm} .
\end{align*}
\end{prop}
} %\HIDE

Thus, we can restrict our attention to the pattern $e_{\pi}$.

\begin{figure}[t]
\centering
\begin{tikzpicture}[scale=0.25]
  \plotc{24}{1,1,3,0,0,3,1,1,2,5,0,1,0,0,0,2,1,0,3,1,1,2,2,0}
\end{tikzpicture}
\caption{The bar-chart representation of a $24$-term composition of 30}\label{figCompBarChart}
\end{figure}
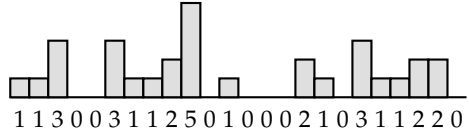

% ================================================================
\section{Compositions and inversion sequences}\label{sectCompositions}

In this section, we introduce two models of random integer compositions.
We then leverage results from~\cite{BTCompositions} concerning patterns in compositions to find the thresholds for consecutive patterns in inversion sequences, and hence also for consecutive patterns in permutations.

An \emph{$n$-term weak composition of $m$}, or just an \emph{$n$-composition of $m$}, is a sequence of $n$ nonnegative integers that sum to~$m$. See Figure~\ref{figCompBarChart} for an example.
If $C$ is a composition, then we use $C(i)$ to denote its $i$th term and $\|C\|$ to denote the sum of its terms (or \emph{weight}).
Let $\CCC_n$ denote the (infinite) set of all $n$-compositions, and let $\ccnm$ denote the set of all $n$-compositions of $m$.
By a simple ``stars and bars'' argument, it can be seen that the total number of distinct $n$-compositions of $m$ is equal to $\binom{m+n-1}{m}$.

We say that a composition $c$ of length $k$ occurs as an \emph{exact pattern} at position $j$ in another composition $C$ if
$C(j-1+i)=c(i)$ for each $i\in[k]$.
Equivalently, $c$ occurs at position $j$ in $C$ if $C[j,j+k-1]=c$, where $C[i,j]$ denotes the sub-composition $C(i),\ldots,C(j)$.
For example, the exact pattern 3112 occurs twice in the composition in Figure~\ref{figCompBarChart}, at positions 6 and 19.

We now present two models of random compositions.
The first is the \emph{uniform random composition} $\cnm$, chosen uniformly from $\ccnm$.
For any $C\in\ccnm$, we have
\[
\prob{\cnm=C} \eq \binom{m+n-1}{m}^{\!\!-1} .
\]

%\HIDE
{
It is easy to see that the distribution of any exact pattern in $\cnm$ is independent of its position.

\begin{prop}\label{propCnmPosInd}
Let $c$ be any exact composition pattern of length $k$. Then,
  for any
  $i,j\in [n+1-k]$,
\[\prob{c \text{ occurs at } i \text{ in } \cnm} \eq \prob{c \text{ occurs at } j \text{ in } \cnm}.\]
\end{prop}

\begin{proof}
The probability of $c$ appearing at position $i$ in $\cnm$ is equal to
\[
{\binom{(m - \|c\|)+(n-k)-1}{m - \|c\|}}\times{\binom{m+n-1}{m}}^{\!\!-1} ,
\]
which doesn't depend on $i$.
\end{proof}
} %\HIDE

%\together9
Our second model, which is significantly easier to analyse, is the \emph{geometric random composition}~$\cnp$.
If $p \in [0,1)$, then $\cnp$ is distributed over $\CCC_n$ so that for each $C \in \CCC_n$, we have \[\prob{\cnp = C} \eq q^np^{\|C\|},\] where $q = 1-p$.
Each term of $\cnp$ is sampled independently from the geometric distribution with parameter $q$, that is, $\prob{C_{n,p}(i) = k} = qp^k$ for each $k \geqs 0$ and $i \in [n]$.
Note that $\expec{\|\cnp\|}=np/q$.
To avoid unnecessary repetition, when considering $\cnp$ in this work, $q$ always denotes $1-p$.

Thresholds are defined in our composition models in an analogous manner to $\snm$.
A function $m^\star = m^\star(n)$ is a threshold for a property $\QQQ$ in the uniform random composition $\cnm$ if
  \[
  \liminfty \prob{\cnm \text{ satisfies }\QQQ} \eq
  \begin{cases}
    0 & \text{if~ $m\ll m^\star$,} \\
    1 & \text{if~ $m\gg m^\star$.}
  \end{cases}
  \]
Similarly, a function $p^\star = p^\star(n)$ is a threshold for a property $\QQQ$ in the geometric random composition $\cnp$ if
\[
  \liminfty \prob{\cnp \text{ satisfies }\QQQ} \eq
  \begin{cases}
    0 & \text{if~ $p/q\ll p^\star/q^\star$,} \\
    1 & \text{if~ $p/q\gg p^\star/q^\star$,}
  \end{cases}
  \]
where $q^\star = 1-p^\star$.

%Our previous work~\cite{BTCompositions} on patterns in compositions enables us to determine the threshold for the appearance of an exact pattern in~$\cnm$.
In our previous work~\cite{BTCompositions} on the evolution of random compositions, we investigate the appearance of patterns in compositions.
We use these results to determine the threshold for the appearance of an exact pattern in~$\cnm$.

\begin{prop}\label{propThreshCnmContainsExactPatt}
If $c$ is a non-zero exact composition pattern of length $k$ with $\|c\|=s$, then for any positive constant~$a$,
\[\liminfty\prob{\cnm \text{ contains }c} \eq \begin{cases}
    0 & \quad \text{if~ } m \ll n^{1-1/s},\\
    1 - e^{-a^{s}} & \quad \text{if~ } m \sim a n^{1-1/s} \text{~and~} s>1,\\
    1 & \quad \text{if~ } m \sim a \text{~and~} s=1,\\
    1 & \quad \text{if~ }  n^{1-1/s} \ll m \ll n^{1+1/k}.
  \end{cases}
\]
\end{prop}

\begin{proof}
This follows from two results in~\cite{BTCompositions},
Proposition~4.1:
\[\liminfty\prob{\cnp \text{ contains } c} \eq \begin{cases}
    0 & \quad \text{if~ } p \ll n^{-1/s},\\
    1 - e^{-a^{s}} & \quad \text{if~ } p \sim a n^{-1/s},\\
    1 & \quad \text{if~ }  n^{-1/s} \ll p \text{ and } q \gg n^{-1/k},
  \end{cases}
\]
and Proposition~4.2:
\[
\text{If~ $m \sim np/q\gg1$ ~then~~}
\prob{\cnm \text{ contains } c} \sim \prob{\cnp \text{ contains } c}.
%\qedhere
\]
If $m<s$, then $\cnm$ doesn't contain~$c$.
If $m$ is bounded and $m\geqs s>1$, then
\[
\prob{\cnm \text{ contains } c}
\;<\;
n\, \frac{\binom{m-s+n-k-1}{m-s}}{\binom{m+n-1}{m}}
\;\sim\;
\frac{m!}{(m-s)!}\, n^{1-s}
\;\ll\; 1 .
\]
If $m\sim a$ and $s=1$, then a.a.s. $\cnm$ contains exactly $a$ occurrences of~$c$, 
this being the same as 
having the first few and last few terms equal to zero, and 
avoiding a finite number of patterns, 
each of weight greater than one,  
whose non-zero terms are close together.
\end{proof}

Clearly an inversion sequence is a special type of composition.
We would like to leverage our results on patterns in compositions in order to establish results concerning inversion sequences, and hence permutations.
To this end we determine when $\cnm$ is a.a.s. an inversion sequence.

\begin{prop}\label{propCnmAnInvSeq}
The threshold for $\cnm$ to be an inversion sequence is given by
\[\liminfty\prob{\cnm \in \eenm} \eq \begin{cases}
    1       & \quad \text{if~ } m \ll n,\\
    0  & \quad \text{if~ } m \gg n.
  \end{cases}
\]
\end{prop}

The proof of this result requires the notion of an \emph{increasing} property.
We say that a property $\QQQ$ of compositions is {increasing} if $C$ satisfying $\QQQ$ implies that $C^{+j}$ satisfies $\QQQ$, for every $j \in [n]$, where $C^{+j}$ denotes the composition obtained from $C$ by the addition of $1$ to its $j$th term.

\begin{proof}
We first establish the threshold for the geometric random composition $\cnp$ to be an inversion sequence.
Recall that $C\in\eenm$ if $C(i)<i$ for each $i \in [n]$.

Now,
$\prob{\cnp(i) < i} \eq 1-p^i.$
So
\[
\prob{\text{$\cnp$ is an inversion sequence}} \eq \prod^{n}_{i=1}\left(1-p^i\right) .
\]
By Euler's Pentagonal Number Theorem (see~\cite{GJ2004}),
\[
%\liminfty
\prod^{\infty}_{i=1}\left(1-p^i\right) \eq 1 + \sum^{\infty}_{k=1}(-1)^k\left(p^{k(3k+1)/2} + p^{k(3k-1)/2}\right) \eq 1 - p - p^2 + p^5 + p^7 -\dots .
\]
If $p \ll 1$, then this converges to 1 as $n$ tends to infinity, and so a.a.s. $\cnp$ is an inversion sequence.

On the other hand,
\[
\prob{\cnp \text{ is an inversion sequence}} \eq q\prod^n_{i=2}\left(1-p^i\right) \;\leqs\; q.
\]
If $q \ll 1$, then this converges to 0 as $n$ tends to infinity, and so a.a.s. $\cnp$ is not an inversion sequence.

Not being an inversion sequence is an {increasing} property.
This enables us to transfer the threshold from $\cnp$ to $\cnm$ by using \cite[Proposition~2.8]{BTCompositions}:
\begin{myquote}
  If $\QQQ$ is an increasing property that has a threshold $p^\star \geqs n^{-1}$ in $\cnp$, then $np^\star/q^\star$ is a threshold for $\QQQ$ in $\cnm$, where $q^\star = 1-p^\star$.
\end{myquote}
The result follows.
\end{proof}

This enables us to handle values of $m\ll n$.
To extend our results to slightly greater $m$, we require the following bound from~\cite{BTCompositions}
on the largest term in $\cnm$.
If $\max(C)$ is the largest term in composition~$C$, then a.a.s. $\max(\cnm)$ does not grow faster than $\frac{m}{n}\log n$.

\begin{prop}[{\cite[Propositions~4.9--4.11 and~2.8]{BTCompositions}}]\label{propCnmMax}
  \[
  \liminfty \prob{\max(\cnm)\gg \frac{m}n\log n} \eq 0 .
  \]
\end{prop}

Using this bound, we can establish that, under suitable conditions, if a pattern a.a.s. occurs in~$\cnm$ then it also a.a.s. occurs in (a suffix of)~$\enm$.

\begin{prop}\label{propCnmToEnm}
  Suppose $c$ is an exact composition pattern,
  and that
  $m^-\gg1$ and $m^+\ll n^2/\log^2n$ are such that
  a.a.s. $\cnm$ contains $c$ whenever $m^-\ll m\ll m^+$.
  Then,
  a.a.s. $\enm$ also contains $c$ under the same conditions on~$m$.
\end{prop}

\begin{proof}
  Suppose $m\ll n^2/\log^2 n$. Then,
  \[
% m\ll \frac{n^2}{\log^2 n} \;\implies\;
  \frac{m}{n}\log n \;\ll\;
  \sqrt{m} \frac{n}{\log n} \frac{\log n}{n}
  \eq \sqrt{m}
  \;\ll\; \frac{n}{\log n}
  \;\ll\; n.
  \]
  Let $k$ satisfy $\frac{m}{n}\log n \ll k \ll \sqrt{m}$.
  Then, by Proposition~\ref{propCnmMax}, a.a.s. no term of $\cnm$ is greater than~$k$.

  Suppose $s\ll m$.
  Then $m^-(n)\ll m\ll m^+(n)$ implies $m^-(n-k)\ll m-s\ll m^+(n-k)$.
  So,
    if a.a.s. $\cnm$ contains $c$ whenever $m^-\ll m\ll m^+$,
    then it is also the case that a.a.s. $\C_{n-k,m-s}$ contains $c$ whenever $m^-\ll m\ll m^+$

  Now consider the suffix $\e'=\enm[m][k+1,n]$ of $\enm$.
  Clearly, $\e'(i)<k+i$ for each $i\in[n-k]$,
  and $m-\binom{k}2\leqs \|\e'\|\leqs m$, with $\binom{k}2\ll m$ by the definition of $k$.

  Hence,
  \begin{align*}
                & \text{a.a.s. $\cnm$ contains $c$ whenever $m^-\ll m\ll m^+$} \\
  \;\implies\;  & \text{a.a.s. $\C_{n-k,\|\e'\|}$ contains $c$ whenever $m^-\ll m\ll m^+$} \\
  \;\implies\;  & \text{a.a.s. $\e'$ contains $c$ whenever $m^-\ll m\ll m^+$} \\
  \;\implies\;  & \text{a.a.s. $\enm$ contains $c$ whenever $m^-\ll m\ll m^+$} ,
  \end{align*}
  as required.
\end{proof}

Having constructed the necessary framework, we now have all we need to prove our first main result, determining the thresholds for consecutive patterns, which we restate here.

\textbf{Theorem~\ref{thmMainResultConsec}.}
\emph{Let $\pi$ be any consecutive permutation pattern of length $k$.
If $s=\inv(\pi)$ and $s'=\inv(\overline{\pi})$, then for any positive constant~$a$,
\begin{align*}
\liminfty\prob{\snm \text{ contains }\pi} &\eq \begin{cases}
    0 & \quad \text{if~ } m \ll n^{1-1/s},\\
    1-e^{-a^{s}} & \quad \text{if~ $m\sim a n^{1-1/s}$} \text{~and~} s>1,\\
    1 & \quad \text{if~ } m \sim a \text{~and~} s=1,\\
    1 & \quad \text{if~ } n^{1 - 1/s} \ll m \ll n^{1+1/k},
  \end{cases}
\\[3pt]
\liminfty\prob{\snm \text{ contains }\pi} &\eq \begin{cases}
    1 & \quad  \text{if~ } n^{1+1/k} \gg \binom{n}{2} - m \gg n^{1-1/s'},\\
    1 & \quad \text{if~ } \binom{n}{2} - m \sim a \text{~and~} s=1,\\
    1-e^{-a^{s'}} & \quad \text{if~ $\binom{n}{2} - m\sim a n^{1-1/s'}$} \text{~and~} s>1, \\
    0 & \quad \text{if~ } \binom{n}{2} - m \ll n^{1-1/s'},
  \end{cases}
\end{align*}
as long as $s>0$ and $s'>0$, respectively.}

\begin{proof}
  If $m \ll n$, then by Proposition~\ref{propCnmAnInvSeq}, a.a.s. $\cnm$ is an inversion sequence.
  So, by Propositions~\ref{propSnmtoEnm}, \ref{propCnmAnInvSeq} and~\ref{propCnmPosInd},
  for any
  $j\in [n+1-|\pi|]$,
  \begin{align*}
\prob{\pi \text{ occurs at position } j \text{ in } \snm}
%&\eq\prob{\pi \text{ occurs at position } 1 \text{ in } \snm} \\
&\eq \prob{e_{\pi} \text{ occurs at position } 1 \text{ in } \enm}\\
&\;\sim\; \prob{e_{\pi} \text{ occurs at position } 1 \text{ in } \cnm}\\
&\eq \prob{e_{\pi} \text{ occurs at position } j \text{ in } \cnm}.
\end{align*}
Therefore $\prob{\snm \text{ contains } \pi} \sim \prob{\cnm \text{ contains } e_{\pi}}$.

From Proposition~\ref{propThreshCnmContainsExactPatt},
if $m\ll n^{1-1/s}$ then a.a.s. $\cnm$ avoids $e_\pi$, and so a.a.s. $\snm$ avoids $\pi$.
The same proposition also gives us the probability at the threshold.

Whenever $n^{1 - 1/s} \ll m \ll n^{1+1/k}$, then, by Proposition~\ref{propThreshCnmContainsExactPatt}, a.a.s. $\cnm$ contains $e_\pi$.
So, by Proposition~\ref{propCnmToEnm}, a.a.s. $\enm$ contains $e_\pi$, and so a.a.s. $\snm$ contains $\pi$.

The threshold for the disappearance of $\pi$ then follows from
the fact that
$\snm[\binom{n}2-m]$ has the same distribution as $\overline{\snm}$.
\end{proof}

\begin{table}[ht]
  \centering\small
  \renewcommand{\arraystretch}{1.15}
  \begin{tabular}{c|l|l|}
                & Consecutive permutation patterns          & Corresponding inversion sequences \\ \hline
    $1$         & 21, 132, 213, 1243, 1324, 2134            & 01, 001, 010, 0001, 0010, 0100  \\
    $\sqrt{n}$  & 231, 312, 1342, 1423, 2143, 2314, 3124    & 002, 011, 0002, 0011, 0101, 0020, 0110  \\
    $n^{2/3}$   & 321, 1432, 2341, 2413, 3142, 3214, 4123   & 012, 0012, 0003, 0021, 0102, 0120, 0111  \\
    $n^{3/4}$   & 2431, 3241, 3412, 4132, 4213              & 0013, 0103, 0022, 0112, 0121  \\
    $n^{4/5}$   & 3421, 4231, 4312                          & 0023, 0113, 0122  \\
    $n^{5/6}$   & 4321                                      & 0123
   %\hline
  \end{tabular}
  \caption{Thresholds for the appearance in $\snm$ of short consecutive patterns}\label{tblConsec}
\end{table}

Thus, $m \sim n^{1-1/s}$ is the threshold for the appearance in $\snm$ of each consecutive permutation pattern (of any length) with $s$ inversions.
See Table~\ref{tblConsec} for patterns of length two, three and four.
So, if $0<\gamma<1$ and $m\sim n^\gamma$, then a.a.s. $\snm$ contains any given consecutive pattern with fewer than $1/(1-\gamma)$ inversions, but avoids any given consecutive pattern with more than $1/(1-\gamma)$ inversions.
Note however that $\snm$ \emph{does} contain a consecutive pattern with $m$ inversions, namely $\snm$ itself!

Consecutive patterns having the same length and number of inversions share thresholds for both appearance and disappearance.
Also, two patterns of different lengths with the same number of inversions share their appearance threshold, but a.a.s. the shorter pattern disappears later than the longer one.
On the other hand, given two patterns of the same length with different numbers of inversions, a.a.s. the one with the fewer inversions both appears and disappears first.
For example, a.a.s. 2143, 32145, 4213, 42315 and 31425 appear in that order, but depart in the order 32145, 42315, 2143, 31425, 4213.
However, we note again that our methods do not enable us to show that $\snm$ a.a.s. contains a given pattern for all values of $m$ between the thresholds for its appearance and disappearance.

%\newpage
% ================================================================
\section{Classical and vincular patterns}\label{sectClassicalVinc}

In this section, we establish thresholds for classical and vincular patterns.
We say that a classical pattern $\pi$ \emph{occurs at $[i,j]$ in $\sigma$} if $\sigma(i)$ is the first term and $\sigma(j)$ the last term in an occurrence of $\pi$.
Such an occurrence has \emph{width} $w=j+1-i$.
We use $\sigma[i,j]$ to denote the permutation of~$[w]$ that has the same relative order as $\sigma(i),\ldots,\sigma(j)$.

We begin with two propositions concerning occurrences of indecomposable classical patterns.

\begin{prop}\label{propIndecompInv}
  Suppose $\alpha$ is an indecomposable classical pattern of length $k\geqs2$.
  If $\alpha$ occurs at $[i,j]$ in a permutation $\sigma$, with width $w=j+1-i$, then $\inv(\sigma[i,j])\geqs\inv(\alpha)+w-k$.
\end{prop}

\begin{proof}
  If $i<\ell<j$ then $\sigma(\ell)$ forms an inversion with some term in the occurrence of $\alpha$.
  Otherwise we would have $\alpha=\beta\oplus\gamma$, with $\beta$ lying to the left and below $\sigma(\ell)$ and $\gamma$ lying to the right and above $\sigma(\ell)$.
  But $\alpha$ is indecomposable.
  Thus each of the $w-k$ terms of $\sigma[i,j]$ not in the occurrence of $\alpha$ contributes at least 1 to the number of inversions in $\sigma[i,j]$.
\end{proof}

With this in hand, we prove that containment of an indecomposable classical pattern implies containment of a consecutive pattern with as many inversions whose length is bounded.
No attempt has been made to optimise the bound on the length.

\begin{prop}\label{propIndecompConsec}
  Suppose $\alpha$ is an indecomposable classical pattern of length $k\geqs2$ with $s$ inversions.
  If $\alpha$ occurs in a permutation $\sigma$, then $\sigma$ contains a consecutive pattern with at least $s$ inversions of length at most $ks$.
\end{prop}

\begin{proof}
Suppose $\alpha$ occurs at $[i,j]$ in $\sigma$ with width $w=j+1-i\geqs k$.
Let $t=\inv(\sigma[i,j])$.
By Proposition~\ref{propIndecompInv}, we have $t\geqs s+w-k$.
Note that $t\geqs s\geqs1$.

Let $d=\floor{t/s}$ and partition $e_{\sigma[i,j]}$ into $d$ consecutive blocks of almost equal length,
each block having length either $\floor{w/d}$ or $\ceil{w/d}$.
Since
$t/d\geqs s$,
%$t/d>s-1$,
by the pigeonhole principle, there is a block $b$ with $\|b\|\geqs s$.

Now, $w\leqs k+t-s$, and
\[
d \eq \floor{\frac{t}s} \;\geqs\; \frac{t}s-1+\frac1s \eq \frac{1+t-s}s.
\]
So the length of each block is bounded above by
\begin{align*}
\ceil{\frac{w}d}
\;<\; \frac{w}d+1
& \;\leqs\; 1+\frac{(k+t-s)s}{1+t-s} \\[6pt]
& \eq 1+s + \frac{(k-1)s}{1+t-s}
\;\leqs\; 1+s + (k-1)s
\eq 1 + ks .
\end{align*}
Thus, since this is a strict inequality, there is a consecutive subsequence of $e_{\sigma[i,j]}$ of length no more than $ks$ whose terms total at least $s$, and so
$\sigma$ contains a consecutive pattern with at least $s$ inversions of length at most~$ks$.
\end{proof}

We are now in a position to establish the thresholds for classical patterns in $\snm$.

\textbf{Theorem~\ref{thmMainResultClassical}.}
\emph{Let $\pi$ be any classical permutation pattern.
If $s$ is the greatest number of inversions in a component of $\pi$,
and $s'$ is the greatest number of inversions in a component of $\overline{\pi}$,
then for any positive constant~$a$,
\begin{align*}
\liminfty\prob{\snm \text{ contains }\pi} &\eq \begin{cases}
    0 & \quad \text{if~ } m \ll n^{1-1/s},\\
    1 & \quad \text{if~ } n^{1 - 1/s} \ll m \ll n,
  \end{cases}
\\[3pt]
\liminfty\prob{\snm \text{ contains }\pi} &\eq \begin{cases}
    1  & \quad  \text{if~ } n \gg \binom{n}{2} - m \gg n^{1-1/s'},\\
    0       & \quad \text{if~ } \binom{n}{2} - m \ll n^{1-1/s'},
  \end{cases}
\end{align*}
as long as $s>0$ and $s'>0$, respectively.}

\begin{proof}
  We first prove that below the threshold a.a.s. $\snm$ avoids $\pi$.
  Indeed, a.a.s. it contains no indecomposable pattern with $s$ inversions.

  By Proposition~\ref{propIndecompConsec}, if $\snm$ were to contain an indecomposable pattern $\alpha$ of length $k$ then it would also contain some consecutive pattern of length at most $ks$ with at least $s$ inversions.
  There are only finitely many such consecutive patterns.
  Now suppose that $m \ll n^{1-1/s}$.
  From Theorem~\ref{thmMainResultConsec}, we know that a.a.s. $\cnm$ contains no fixed finite set of consecutive patterns with $s$ or more inversions.
  Thus $\cnm$ avoids $\alpha$, and hence also avoids $\pi$.

  We now prove that above the threshold a.a.s. $\snm$ contains $\pi$.
  Suppose $\pi$ has sum decomposition $\pi=\alpha_1\oplus\ldots\oplus\alpha_r$. % $r$ (indecomposable) components.

  Let $\C=\cnm$.
  For $0\leqs j\leqs r$, let $i_j=\floor{jn/r}$, and, for each $j\in[r]$, let $\C_j=\C[i_{j-1}+1,i_j]$.
  Thus, $\C_1,\ldots,\C_r$ is a partition of the terms of $\C$, each $\C_j$ having length $n_j\in\big\{ \!\floor{n/r}, \ceil{n/r}\! \big\}$.
  Let $m_j=\|\C_j\|$.

  Since $\|\C\|$ is constant,
  the covariance %$\cov{\C(i),\C(j)}$
  between any two distinct terms of $\C$ is negative.
  Indeed, straightforward calculations show that %if $i_1\neq i_2$, then
  \[
  \var{\C(i)} \eq \frac{(n-1)m(m+n)}{n^2(n+1)} ,
  \qquad \text{and} \qquad
  \cov{\C(i_1),\C(i_2)} \eq -\frac{m(m+n)}{n^2(n+1)} \text{~~if $i_1\neq i_2$}.
  \]
  %which tends to zero as long as $m\ll n^{3/2}$.
  \together5
  Hence,
  \[
  %\varbig{\frac{m_j}{n_j}}
  \var{{m_j}/{n_j}}
  \eq {\var{m_j}}/{n_j^2}
 %\;\sim\; \frac1{n^2}\! \left(\frac{n}r\, \var{\C(i)} + \frac{n}r \Big(\frac{n}r-1\Big) \cov{\C(i_1),\C(i_2)}\right)
  \;<\; \, n_j\, \var{\C(i)} / n_j^2
 %\eq \frac{(r-1)m(m+n)}{r^2n^2(n+1)} ,
  \;\sim\; \frac{rm(m+n)}{n^3} ,
  \]
  which tends to zero as long as $m\ll n^{3/2}$.
  Thus (by Chebyshev's inequality),
  for this range of values for $m$
 %this is sufficient for
  the sum of terms in each $\C_j$ satisfies a law of large numbers.
  %(Bernstein~\cite{Bernstein1918}; see~\cite{KMS2004} and~\cite[Problem 3.3.2]{Shiryaev2012}).

  Thus, for each~$j$ and any $\veps>0$, a.a.s. we have
 %$(1-\veps)m/r<m_j<(1+\veps)m/r$.
  $m_j>(1-\veps)m/r$.
  Therefore, if
 %$n^{1 - 1/s} \ll m \ll n$,
  $m\gg n^{1 - 1/s}$,
  then
 %$n_j^{1 - 1/s} \ll m_j \ll n_j$ for each $j\in[r]$.
  $m_j\gg n_j^{1 - 1/s}$ for each $j\in[r]$.

  \together2
  Thus, if
  $n^{1 - 1/s} \ll m \ll n$, for each $j\in[r]$, we have the following sequence of implications:
  \begin{bullets}
    \item[$\bullet$] By Proposition~\ref{propThreshCnmContainsExactPatt}, a.a.s. $\C_{n_j,m_j}$ contains a consecutive occurrence of $e_{\alpha_j}$.
    \item[$\bullet$] Thus a.a.s. $\C=\cnm$ contains consecutive occurrences of $e_{\alpha_1},\ldots e_{\alpha_r}$ in that order.
    \item[$\bullet$] Since, by Proposition~\ref{propCnmAnInvSeq}, $\cnm$ is a.a.s. an inversion sequence, a.a.s. $\enm$ contains consecutive occurrences of $e_{\alpha_1},\ldots e_{\alpha_r}$ in that order.
    \item[$\bullet$] By Proposition~\ref{propPosjinInvPosjinPerm}, these correspond to occurrences of $\alpha_1,\ldots,\alpha_r$ as consecutive patterns in $\snm$, such that no point of $\snm$ is to the upper left of any point in any of these occurrences.
    \item[$\bullet$] Thus a.a.s. $\pi=\alpha_1\oplus \ldots \oplus \alpha_r$ occurs in $\snm$.
  \end{bullets}

The threshold for the disappearance of $\pi$ then follows
because
$\snm[\binom{n}2-m]$ has the same distribution as $\overline{\snm}$.
\end{proof}

Unlike with consecutive patterns, classical patterns having the same length and number of inversions need share neither threshold. % for appearance nor threshold for disappearance.
See Table~\ref{tblClassical2} for an illustration of this.

\begin{table}[ht]
  \centering\small
  \renewcommand{\arraystretch}{1.3}
  \begin{tabular}{c|c|c}
              & Pattern &   \\ \hline
    $n^{2/3}$ & 321654  & $\binom{n}2-n^{8/9}$  \\
    $n^{4/5}$ & 423165  & $\binom{n}2-n^{8/9}$  \\
    $n^{8/9}$ & 561324  & $\binom{n}2-n^{4/5}$  \\
    $n^{8/9}$ & 456123  & $\binom{n}2-n^{2/3}$  \\
    \hline
  \end{tabular}
  \caption{Thresholds in $\snm$ for the appearance and disappearance of four classical patterns of length six with six inversions}\label{tblClassical2}
\end{table}

We can say a little more about the appearance of classical patterns.
As the random permutation $\snm$ evolves, indecomposable patterns appear first as consecutive patterns.
Let us say that a component $\alpha$ of a classical pattern $\pi$ is \emph{dominant} if no other component of $\pi$ has more inversions than $\alpha$.
Then for an arbitrary classical pattern, we have the following behaviour.
\begin{prop}
  Let $\pi$ be any classical permutation pattern whose dominant components have $s$ inversions.
  If $n^{1 - 1/s} \ll m\ll n^{1 - 1/(s+1)}$, then
  asymptotically almost surely, in every occurrence of $\pi$ in~$\snm$, each dominant component of $\pi$ occurs consecutively.
\end{prop}
\begin{proof}
  Suppose $\alpha$ is a dominant component of $\pi$.
  By Propositions~\ref{propIndecompInv} and~\ref{propIndecompConsec}, if there is a nonconsecutive occurrence of $\alpha$ in $\snm$, then there is an occurrence of some consecutive pattern with at least $s+1$ inversions.
  If $m\ll n^{1 - 1/(s+1)}$ than a.a.s. this does not occur.
\end{proof}

A direct application of Theorems~\ref{thmMainResultConsec} and~\ref{thmMainResultClassical} yields the threshold for the longest decreasing subsequence in $\snm$ to have a given length.
\begin{cor}\label{corLDS}
If $\ell\geqs2$, then
\[
%\liminfty\prob{\text{$\snm$ contains a decreasing subsequence of length $\ell$}} \eq \begin{cases}
\liminfty\prob{\text{$\snm$ contains $\ell\ldots21$}} \eq \begin{cases}
    0 & \quad \text{if~ } m \ll n^{1-1/\binom{\ell}{2}},\\
    1 & \quad \text{if~ } n^{1 - 1/\binom{\ell}{2}} \ll m \ll n^{1+1/\ell} .
  \end{cases}
\]
\end{cor}
Thus, if $0<\gamma<1$ and $m\sim n^\gamma$, then the length $\ell$ of the longest decreasing subsequence in $\snm$ depends on the value of the solution of the equation $\gamma = 1-1/\binom{d}{2}$:
\[
d \eq \frac12\!\left(\!1+\sqrt{\frac{9-\gamma}{1-\gamma}}\,\right) \! .
\]
If $d$ is not an integer, then a.a.s. $\ell=\floor{d}$, and every longest decreasing subsequence occurs consecutively.
If $d$ is an integer, then a.a.s. $\ell$ takes one of the two values $d-1$ or $d$.

%\together5
We can also determine the threshold for the occurrence of an inversion with a given width.

\begin{cor}
If $w\geqs2$, then
\[
\liminfty\prob{\text{$\snm$ contains an inversion of width $w$}} \eq \begin{cases}
    0 & \quad \text{if~ } m \ll n^{1-1/(w-1)},\\
    1 & \quad \text{if~ } n^{1 - 1/(w-1)} \ll m \ll n^{1+1/w} .
  \end{cases}
\]
\end{cor}
\begin{proof}
  By Proposition~\ref{propIndecompInv}, if an inversion of width $w$ occurs at $[i,j]$ in $\sigma$ then $\sigma[i,j]$ has at least $w-1$ inversions.
  So, by Theorem~\ref{thmMainResultConsec}, below the threshold a.a.s. there is no occurrence of an inversion of width $w$ in $\snm$.
  On the other hand, the first and last terms of consecutive pattern $\pi=234\ldots w1$ %--- with inversion sequence $000\ldots0(w-1)$ ---
  form an inversion of width $w$. Since $\inv(\pi)=w-1$, by Theorem~\ref{thmMainResultConsec}, above the threshold a.a.s. $\snm$ contains an occurrence of $\pi$ and hence also contains an inversion of width~$w$.
\end{proof}

Another consequence of Theorem~\ref{thmMainResultClassical} is that if $m$ grows sufficiently fast (e.g. $m=n/\log n$), then a.a.s. $\snm$ contains \emph{any} given classical pattern, or equivalently a.a.s. $\snm$ is not contained in any permutation class $\av(B)$ --- consisting of permutations avoiding the classical patterns in the set~$B$.

\begin{cor}
  If $n^{1-\delta}\ll m \ll n$ for every $\delta>0$, and $\pi$ is any classical permutation pattern, then asymptotically almost surely $\snm$ contains $\pi$.
\end{cor}

The thresholds for vincular patterns generalise those for consecutive and classical patterns.
Recall the definition of a supercomponent of a vincular pattern from page~\pageref{defVinc}.

\textbf{Theorem~\ref{thmMainResultVinc}.}
\emph{Let $\pi$ be any vincular permutation pattern.
If $s$ is the greatest number of inversions in a supercomponent of $\pi$,
and $s'$ is the greatest number of inversions in a supercomponent of $\overline{\pi}$,
then for any positive constant~$a$,
\begin{align*}
\liminfty\prob{\snm \text{ contains }\pi} &\eq \begin{cases}
    0 & \quad \text{if } m \ll n^{1-1/s},\\
    1 & \quad \text{if } n^{1 - 1/s} \ll m \ll n,
  \end{cases}
\\[3pt]
\liminfty\prob{\snm \text{ contains }\pi} &\eq \begin{cases}
    1  & \quad  \text{if } n \gg \binom{n}{2} - m \gg n^{1-1/s'},\\
    0  & \quad \text{if } \binom{n}{2} - m \ll n^{1-1/s'},
  \end{cases}
\end{align*}
as long as $s>0$ and $s'>0$, respectively.}

\begin{proof}
  The proof is entirely analogous to that of Theorem~\ref{thmMainResultClassical}, with ``supercomponent'' replacing ``component''.
  Versions of Propositions~\ref{propIndecompInv} and~\ref{propIndecompConsec} are also required with ``supercomponent'' replacing ``indecomposable classical pattern''.
  We leave the (straightforward) details to the reader.
 %The key observation is the following:
 %Suppose $\beta$ is a supercomponent with sum decomposition $\beta=\alpha_1\oplus \ldots \oplus\alpha_r$.
 %If $\beta$ occurs at $[i,j]$ in $\sigma$.
\end{proof}

%\rulebreak

% ================================================================
%\section{Mallows permutations}
%\together9
\section{Open questions}\label{sectGapMallows}

As noted in the introduction, our methods do not enable us to show that $\snm$ a.a.s. contains a given pattern for all values of $m$ between the thresholds for its appearance and disappearance.
This is because our approach builds on results concerning random compositions,
%and $\cnm$ is a.a.s. an inversion sequence only if $m\ll n$ (Proposition~\ref{propCnmAnInvSeq}).
and the threshold for the disappearance of an exact pattern of length $k$ in $\cnm$ is $m\sim n^{1+1/k}$ (Proposition~\ref{propThreshCnmContainsExactPatt}).
We also remarked that even the asymptotics of $|\ssnm|$ appear not to have been established when $n\ll m\ll n^2$.
When $m\sim a n^2$, we do know a bit more.
Specifically, the permuton approach (see~\cite{KKRW2020}) is sufficient to establish that any \emph{classical} pattern is present a.a.s.
However, permutons tell us nothing about the local structure (see~\cite{BevanScaleLimits,BP2020}), so cannot help with consecutive patterns.
\begin{prop}
  Let $\pi$ be any classical permutation pattern, and suppose $m\sim an^2$ for some constant $a\in(0,\frac12)$.
  Then $\liminfty \prob{\text{$\snm$ contains $\pi$}}=1$.
\end{prop}

Nevertheless, there seems no reason to doubt that any pattern is a.a.s. present in $\snm$ between its two thresholds.
Somewhat surprisingly, there appears to be no simple strategy for proving this.
\begin{conj}
  Let $\pi$ be any consecutive permutation pattern,
  and let $s=\inv(\pi)$ and $s'=\inv(\overline{\pi})$.
  If $n^{1 - 1/s} \ll m$ and $\binom{n}{2} - m \gg n^{1-1/s'}$,
  then
  $\liminfty\prob{\snm \text{ contains }\pi} = 1$.
\end{conj}

Furthermore, it is natural to suppose that the expected number of occurrences in $\snm$ of a given pattern is a unimodal function of $m$, first increasing and then decreasing.
Given a permutation pattern $\pi$, let $\bbE_{n,m}(\pi)$ denote the expected number of occurrences of $\pi$ in~$\snm$.
\begin{conj}\label{conjSnmUnimodal}
  Let $\pi$ be any consecutive permutation pattern.
  If $n\geqs |\pi|+2$, then the sequence
  $
  \bbE_{n,0}(\pi) ,\, \bbE_{n,1}(\pi) ,\,  \ldots  ,\, \bbE_{n,{\binom{n}2}}(\pi)
  $
  is unimodal.
\end{conj}

Here are some further questions motivated by our considerations.

\begin{question}
  What are the asymptotics of $|\ssnm|$ when $n\ll m\ll n^2$?
\end{question}

\begin{question}[{see Acan and Pittel~\cite{AP2013}}]
  Is it possible to define an \emph{explicit} Markov process that produces $\snm[m+1]$ from $\snm$?
  Can this be achieved in a natural way?
\end{question}

%Extending Corollary~\ref{corLDS}
\begin{question}
  What is the length of the longest decreasing subsequence in $\snm$ for ranges of $m\gg n^{1-\delta}$ for every $\delta>0$?
\end{question}

The following two questions were addressed unsuccessfully in~\cite{BevanLocallyUniform}.
\begin{question}
  Given $k=k(n)$, what is the threshold in $\snm$ for each consecutive pattern $\pi$ of length~$k$ to be asymptotically equally likely to occur? That is, for each $\pi\in\SSS_k$ to satisfy
  \[
  \prob{\text{$\pi$ occurs at position 1 in $\snm$}} \sim 1/k! .
  \]
\end{question}

\begin{question}
  Given $w=w(n)$, what is the threshold for the two terms $\snm(1)$ and $\snm(w)$ to be equally likely to form an inversion (of width $w$)
  as not?
  That is, for
  \[
  \liminfty \prob{\snm(1)>\snm(w)} \eq \tfrac12 .
  \]
\end{question}

% --------------------------------
\subsection*{Mallows permutations}
%\together5
%\textbf{Consecutive patterns}

One model that may help in addressing such questions is the \emph{Mallows permutation}~\cite{Mallows1957}, introduced as a statistical model for ranking data.\footnote{It is conventional to use $q$ for the Mallows parameter. For consistency with the other models discussed, we have chosen to use $p$ here.}
The Mallows distribution on $\SSS_n$ with parameter $p\in[0,\infty)$ --- which we denote $\snp$ --- assigns to each $\sigma\in\SSS_n$ the probability
\[
\prob{\snp=\sigma} \eq \frac{(1-p)^n p^{\inv(\sigma)}}{\prod_{j=1}^n(1-p^j)} .
\]
Thus, in the \emph{Mallows inversion sequence} $\enp=e_{\snp}$,
each term independently satisfies a truncated geometric distribution.
Specifically, for each $j\in[n]$ and $k=0,\,\ldots,\,j-1$,
\[
\prob{\enp(j)=k} \eq
\begin{cases}
  \dfrac{(1-p) p^k}{1-p^j}, & \text{if~ $p\notin\{0,1\}$}        , \\[10pt]
  1,                    & \text{if~ $p=0$  and $k=0$}        , \\[3pt]
  0,                    & \text{if~ $p=0$  and $k>0$}        , \\[5pt]
  1/j,             & \text{if~ $p=1$}.
\end{cases}
\]
If $p=1$, then $\snp$ is simply a uniformly chosen random $n$-permutation $\s_n$.
Moreover, the identity $\snp[1/p]=\overline{\snp}$ enables us to restrict our attention to $p\leqs1$.

Of particular relevance to our concerns is the work of Crane and DeSalvo~\cite{CDeS2018} %\cite[Theorems~3.8 and~3.9]{CDeS2018}
on consecutive pattern avoidance in $\snp$ (see also~\cite{CDeSE2018}).
Other research of interest (some restricted to constant~$p$) includes
  the longest increasing subsequence in $\snp$~\cite{BB2017,BP2015,MS2013a},
  growth rates for Mallows permutations avoiding classical patterns~\cite{Pinsky2021},
  the number of descents (consecutive 21 patterns) in $\snp$~\cite{He2022},
  and work on Mallows processes~\cite{Corsini2022}.

As with $\snm$ (see Conjecture~\ref{conjSnmUnimodal}), it is to be expected that the mean number of occurrences in $\snp$ of a given pattern is unimodal in $p$.
Given a permutation pattern $\pi$, let $\bbE_{n,p}(\pi)$ denote the expected number of occurrences of $\pi$ in~$\snp$.
\begin{conj}
  If $\pi$ is any consecutive permutation pattern, then
  the function $\bbE_{n,p}(\pi)$
  is unimodal in~$p$.
\end{conj}

However, the relationship between $\snp$ and $\snm$ appears to be poorly understood from an evolutionary perspective.
Indeed, only very recently has the expected number of inversions of a Mallows permutation been published for the full range of values of $p=p(n)$.
Note that $\expec{\|\enp\|}\sim \expec{\|\cnp\|}$ if this expectation grows subquadratically.
\together9
\begin{prop} % [Propositions 1.2 \& 1.3]
  If $0\leqs p<1$ and $a>0$ is constant, then
  \[
  \expec{\inv(\snp)} \;\sim\;
  \begin{cases}
  \frac{p}{1-p} n,                  & \text{if~ $1-p \gg n^{-1}$}             , \\[5pt]
  c(a)n^2,                     & \text{if~ $1-p \sim {a}n^{-1}$}    , \\[5pt]
  \frac14 n^2,                      & \text{if~ $1-p \ll n^{-1}$}           , \\[5pt]
  \end{cases}
  \]
  with
  \[
  c(a) \eq %\frac1{a}\left( 1 - \frac{\pi^2}{6a} - \log(1-e^{-a}) + \frac1{a}{\mathrm{Li}_2(e^{-a})} \right)
  \frac1{a} - \frac1{a^2}\mathrm{Li}_2(1-e^{-a}) ,
  \]
  where $\mathrm{Li}_2(z)$ is the dilogarithm function. The constant $c(a)$
  satisfies $\lim\limits_{a\to0}c(a)=\frac14$, and for large $a$, we have $c(a) \sim 1/a$.
\end{prop}

\together5
\begin{proof}
  Let $q=1-p$.
  When $q\gg n^{-1}$, we show that $\expec{\enp(j)}$ is close to $p/q$ for most values of~$j$.

  Suppose first that $p\ll1$, and $j > 2+\log_{1/p}(n+p^2)$. Then,
  \[
    \expec{\enp(j)}
    \eq \sum_{k=0}^{j-1} \dfrac{k q p^k}{1-p^j}
    \eq \frac{p}q - \frac{j p^j}{1-p^j}
    \;\geqs\; \frac{p}q - \frac{n p^j}{1-p^j}
    \;>\; \frac{p}q - p^2 .
  \]
  Thus, $\big(1-o(1)\big)np/q < \expec{\|\enp\|} < np/q$.

  Now suppose that $q={\omega}/n$, where $1\ll \omega\ll n$, so that $n^{-1}\ll q\ll1$ and $p\sim1$.
  Suppose also that
  \[
  j \;>\; \frac{\log(nq\omega+1)}{-\log(1-q)} \;\sim\; 2n\frac{\log\omega}{\omega} .
  \]
  Then,
  \begin{align*}
  \expec{\enp(j)}
    \eq \sum_{k=0}^{j-1} \dfrac{k q p^k}{1-p^j}
  & \eq q^{-1}-1-j\left(\frac{1}{1-(1-q)^j}-1\right) \\
  & \;\geqs\; q^{-1}-1-n\left(\frac{1}{1-(1-q)^j}-1\right)
  \;>\; q^{-1}-1-\frac{q^{-1}}{\omega} .
  \end{align*}
  Thus, $\big(1-o(1)\big)n/q < \expec{\|\enp\|} < n/q$.

  Proofs for the other ranges of values of $p$ ($p$ constant, $q\sim an^{-1}$ and $q\ll n^{-1}$) can be found in Pinsky~\cite[Proposition~1.2 and Theorem~1.3]{Pinsky2022}.
\end{proof}

The situation for random permutations is in contrast with that for random graphs~\cite{Bollobas2001,FK2015,JLR2000}, where the connection between results concerning $\gnp$ and those concerning $\gnm$ has been very well studied.
%Similarly for random compositions~\cite{BTCompositions}, results concerning $\cnp$ ... $\cnm$
The relationship between $\snp$ and $\snm$ deserves further investigation.

%\rulebreak
%\newpage

% ================================================================
\section{Appendix: inversion sequence inequalities}

In this appendix, we establish the correspondence between an occurrence of a consecutive pattern in a permutation $\sigma$ and the inequalities satisfied by terms of $e_\sigma$.
If $\pi(i)=h$ then we write $\pi^{-1}(h) = i$.

\begin{prop}\label{propIneqs}
  If $\pi$ is a consecutive pattern of length $k$, then $\pi$ occurs at position $j$ in a permutation~$\sigma$ if and only if, for each $h\in[k-1]$,
  \[
        e_\sigma\big(j-1+\pi^{-1}(h)\big) \:-\: e_\pi\big(\pi^{-1}(h)\big) \:\:\geqs\:\: e_\sigma\big(j-1+\pi^{-1}(h+1)\big) \:-\: e_\pi\big(\pi^{-1}(h+1)\big) .
  \]
\end{prop}

For example, as illustrated in Figure~\ref{figIneqs}, the consecutive pattern $4132$ occurs at position 7 in a permutation~$\sigma$ if and only~if
\[
e_\sigma(8) - 1 \:\:\geqs\:\:
e_\sigma(10) - 2 \:\:\geqs\:\:
e_\sigma(9) - 1 \:\:\geqs\:\:
e_\sigma(7) %- 0
.
\]

\begin{proof}
  To abbreviate, for each $i\in[k]$, let $e_\sigma^{j+}(i)=e_\sigma(j-1+i)$.
  Suppose $\pi$ occurs at position $j$ in~$\sigma$.
  For each $i\in[k]$, let $P_i$ be the $i$th point from the left in the occurrence of $\pi$ in~$\sigma$.
  Then, by definition, for each $i$, the inversion sequence entry $e_\sigma^{j+}(i)$ is the number of points in $\sigma$ above and to the left of~$P_i$.
  Also, $e_\pi(i)$ is the number of points above and to the left of $P_i$ in the occurrence of~$\pi$.
  Thus, $N(i):=e_\sigma^{j+}(i)-e_\pi(i)$ is the number of points above $P_i$ that are to the left of the occurrence of~$\pi$.

\begin{figure}[t]
  \centering
  \begin{tikzpicture}[scale=0.25]
    \plotperm{13}{5,13,12,2,9,4,0,0,0,0,10,1,8}
   %\plotpermnobox[blue]{}{0,0,0,0,0,0,11,3,7,6}
    \circpt{7}{11}
    \circpt{8}{3}
    \circpt{9}{7}
    \circpt{10}{6}
    \draw[thin,gray!75!blue] (6.5,0.5)--(6.5,13.5);
    \draw[thin,gray!75!blue] (0.5,3)--(6.5,3);
    \draw[thin,gray!75!blue] (0.5,6)--(6.5,6);
    \draw[thin,gray!75!blue] (0.5,7)--(6.5,7);
    \draw[thin,gray!75!blue] (0.5,11)--(6.5,11);
    \node at (-2.25,4.5)   {$a_1=2$};
    \node at (-2.25,6.5)   {$a_2=0$};
    \node at (-2.25,9)     {$a_3=1$};
   %\node at (-2.25,12.25) {$a_4=2$};
    \node at (16.25,3)     {\phantom{$a_0=0$}};
  \end{tikzpicture}
  \caption{An occurrence of the consecutive pattern $4132$ at position 7 in a permutation}\label{figIneqs}
\end{figure}
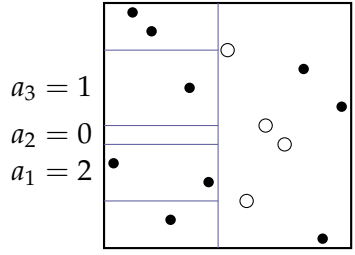

  Now, for each $h\in[k]$, let $P'_h=P_{\pi^{-1}(h)}$ be the $h$th point from the {bottom} in the occurrence of~$\pi$.
  Also, if $h\neq k$, let $a_h=N\big(\pi^{-1}(h)\big) - N\big(\pi^{-1}(h+1)\big)$, which is thus the number of points to the left of the occurrence of $\pi$ that lie above $P'_h$ but below $P'_{h+1}$.
  See Figure~\ref{figIneqs} for an example, each $a_h$ being the number of points in the adjacent rectangle.

  Since each $a_h$ counts the number of points in a well-defined region of $\sigma$, each $a_h$ is nonnegative.
  But $a_h$ equals the difference between the left and right hand sides of the inequality in the statement of the proposition,
  so this inequality holds for each $h\in[k-1]$ as required. % (with equality if~$a_h=0$).

  Suppose now that $\rho$ is a consecutive pattern of length $k$ that is distinct from $\pi$.
  We need to prove that the series of inequalities satisfied by the entries in $e_\sigma$ when $\pi$ occurs at position~$j$ in~$\sigma$ are never all satisfied when $\rho$ occurs at position~$j$.

  Suppose $s$ is the least index such that the relative order of the first $s$ points of $\pi$ differs from that of the first $s$ points of~$\rho$.
  Let $r<s$ be an index such that the relative order of $\pi(r)\pi(s)$ differs from that of~$\rho(r)\rho(s)$.
  We may assume that $\pi(r)<\pi(s)$ but~$\rho(r)>\rho(s)$.

%\together{12}
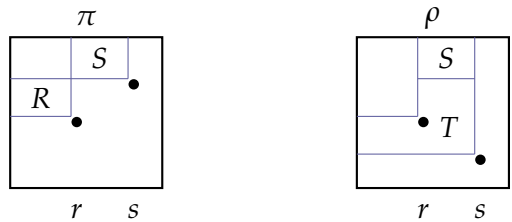
\begin{figure}[ht]
  \centering
  \begin{tikzpicture}[scale=0.25]
    \plotperm{8}{0,0,0,4,0,0,6}
    \node at (4.5,9.5){$\pi$};
    \node at (4,-.75)   {$r$};
    \node at (7,-.75)   {$s$};
    \node at (2.1,5.3)  {$R$};
    \node at (5.2,7.4)  {$S$};
    \draw[thin,gray!75!blue] (0.5,4.3)--(3.7,4.3);
    \draw[thin,gray!75!blue] (0.5,6.3)--(6.7,6.3);
    \draw[thin,gray!75!blue] (3.7,4.3)--(3.7,8.5);
    \draw[thin,gray!75!blue] (6.7,6.3)--(6.7,8.5);
  \end{tikzpicture}
  $\qquad\qquad\qquad$
  \begin{tikzpicture}[scale=0.25]
    \plotperm{8}{0,0,0,4,0,0,2}
    \node at (4.5,9.5){$\rho$};
    \node at (4,-.75)   {$r$};
    \node at (7,-.75)   {$s$};
    \node at (5.2,7.4)  {$S$};
    \node at (5.35,3.75){$T$};
    \draw[thin,gray!75!blue] (0.5,2.3)--(6.7,2.3);
    \draw[thin,gray!75!blue] (0.5,4.3)--(3.7,4.3);
    \draw[thin,gray!75!blue] (3.7,6.3)--(6.7,6.3);
    \draw[thin,gray!75!blue] (3.7,4.3)--(3.7,8.5);
    \draw[thin,gray!75!blue] (6.7,2.3)--(6.7,8.5);
  \end{tikzpicture}
  \caption{An illustration of the second part of the proof of Proposition~\ref{propIneqs}}\label{figIneqsProof}
\end{figure}

  Let $R$ be the number of points in $\pi$ above and to the left of %(the point corresponding to)
  $\pi(r)$ that are not above and to the left of~$\pi(s)$, and
  let $S$ be the number of points in $\pi$ above and to the left of $\pi(s)$ that are not above and to the left of~$\pi(r)$.
  See the left of Figure~\ref{figIneqsProof} for an illustration.
  Now, if $\pi$ occurs at position $j$ in $\sigma$, then
  \[
  e_\sigma^{j+}(r) \:-\: R \:\:\geqs\:\: e_\sigma^{j+}(s) \:-\: S .
  \]
  Now consider the points in $\rho$ above and to the left of $\rho(s)$ that are not above and to the left of~$\rho(r)$.
  Note that this includes every point counted by $S$, because the first $s-1$ points of $\pi$ and the first $s-1$ points of $\rho$ have the same relative order.
  Let $S+T$ be their total number.
  See the right of Figure~\ref{figIneqsProof} for an illustration.
  Now, if $\rho$ occurs at position $j$ in $\sigma$, then
  \[
  e_\sigma^{j+}(s) \:-\: (S+T) \:\:\geqs\:\: e_\sigma^{j+}(r) .
  \]
  Note that $T\geqs1$ because $\rho(r)$ is itself counted by $T$.

  \together5
  Thus, if $\pi$ occurs at position $j$ in $\sigma$, then
  \[
  e_\sigma^{j+}(r) \:-\:  e_\sigma^{j+}(s) \:\:\geqs\:\: -S \:+\: R \:\:\geqs\:\: -S .
  \]
  However, if $\rho$ occurs at position $j$ in $\sigma$, then
  \[
  e_\sigma^{j+}(r) \:-\:  e_\sigma^{j+}(s) \:\:\leqs\:\: -S \:-\: T \:\:<\:\: -S ,
  \]
  which is the negation of the inequality for an occurrence of~$\pi$.
  Hence, for any host permutation~$\sigma$, there is at least one inequality satisfied by the entries in $e_\sigma$ when $\pi$ occurs at position~$j$ in~$\sigma$ that is not satisfied when $\rho$ occurs at position~$j$, and \emph{vice versa}.
\end{proof}

% ================================================================
\subsection*{Acknowledgements}
The authors are grateful to Ross Pinsky and to an anonymous referee for comments that led to improvements to the paper.

%\newpage
% ================================================================
\bibliographystyle{plain}
{\footnotesize\bibliography{../bib/mybib}}

\end{document}